\documentclass[11pt, oneside]{amsart}
\usepackage{geometry}
\usepackage{graphicx}
\usepackage{amssymb}
\usepackage{amsthm}
\usepackage{enumerate}
\usepackage{algorithm,algorithmic}
\usepackage[color=lightgray]{todonotes}

\newcommand{\conv}{\operatorname{conv}}
\newcommand{\id}{\mathbb{I}}

\newcommand{\orderO}{\mathcal{O}}
\newcommand{\scalprod}[2]{{#1}^\intercal #2}

\newcommand{\tw}{\operatorname{tw}}

\newcommand{\R}{\mathbb{R}}
\newcommand{\Z}{\mathbb{Z}}
\newcommand{\Q}{\mathbb{Q}}

\newcommand{\norm}[1]{\left\lVert#1\right\rVert}

\newcommand{\gap}{gap}
\newcommand{\Gap}{Gap}

\newtheorem{thm}{Theorem}
\newtheorem{prop}[thm]{Proposition}
\newtheorem{lem}[thm]{Lemma}

\newtheorem{cor}[thm]{Corollary}

\newcommand{\pitch}{p}
\newcommand{\forbiddenvtxgraph}[1]{H[\bar{#1}]}

\title{Characterizing Polytopes in the $0/1$-Cube with Bounded Chv\'atal-Gomory Rank}

\author{Yohann Benchetrit \and Samuel Fiorini \and Tony Huynh \and Stefan Weltge}

\begin{document}

\begin{abstract}
Let $S \subseteq \{0,1\}^n$ and $R$ be any polytope contained in $[0,1]^n$ with $R \cap \{0,1\}^n = S$.  
We prove that $R$ has bounded Chv\'atal-Gomory rank (CG-rank) provided that $S$ has bounded \emph{notch} and bounded \emph{gap}, where the notch is the minimum integer $\pitch$ such that all $\pitch$-dimensional faces of the $0/1$-cube have a nonempty intersection with $S$, and the gap is a measure of the size of the facet coefficients of $\conv(S)$. 

Let $\forbiddenvtxgraph{S}$ denote the subgraph of the $n$-cube induced by the vertices not in $S$.  We prove that if $\forbiddenvtxgraph{S}$ does not contain a subdivision of a large complete graph, then both the notch and the gap are bounded. By our main result, this implies that the CG-rank of $R$ is bounded as a function of the treewidth of $\forbiddenvtxgraph{S}$. We also prove that if $S$ has notch $3$, then the CG-rank of $R$ is always bounded. Both results generalize a recent theorem of Cornu\'ejols and Lee \cite{CL2016}, who proved that the CG-rank is bounded by a constant if the treewidth of $\forbiddenvtxgraph{S}$ is at most $2$.    

\end{abstract}

\maketitle

\section{Introduction}

Given a polytope $ R \subseteq \R^n $, its first \emph{Chv\'atal-Gomory-closure} (\emph{CG-closure}) is defined as $ R' := \{ x \in \R^n : \scalprod{c}{x} \geqslant \lceil \min_{y \in R} \scalprod{c}{y} \rceil \ \forall \, c \in \Z^n \} $, which can be shown to be again a (rational) polytope  with $R' \cap \Z^n = R \cap \Z^n$ (see Dadush, Dey, and Vielma~\cite{DDV2014}).
By setting $ R^{(0)} := R $ and $ R^{(t)} := (R^{(t-1)})' $ for every $ t \in \Z_{\geqslant 1} $, one recursively defines the $ t $-th CG-closure $ R^{(t)} $ of $ R $.
Chv\'atal~\cite{Chvatal73} proved that there exists a number $ t \in \Z_{\geqslant 0} $ such that $ R^{(t)} = \conv(R \cap \Z^n) $, and the smallest such number is called the \emph{Chv\'atal-Gomory-rank} (\emph{CG-rank}) of $ R $.
In this paper, we give new bounds on the CG-rank of a polytope $ R $ contained in $ [0,1]^n $ that only depend on properties of $ S = R \cap \{0,1\}^n $ and not on $ R $ itself.
This is in the spirit of Conforti, Del Pia, Di Summa, Faenza, and Grappe~\cite{CdPdSFG2015} except that we only consider relaxations contained in $ [0,1]^n $.

One particular reason to study the CG-rank is to obtain bounds on lengths of \emph{cutting-plane proofs} as introduced by Chv{\'a}tal, Cook, and Hartmann~\cite[Sec.~6]{CCH1989}.
Letting $ k $ be the CG-rank of $ R \subseteq \R^n $, the length of a cutting-plane proof is at most $ (n^{k+1} - 1) / (n - 1) $.
In fact, if $k$ is a fixed constant and $R \subseteq \R^n$ has CG-rank $k$, then optimizing a linear function over $R \cap \Z^n$ is one of the few problems that is known to be in $\mathsf{coNP} \cap \mathsf{NP}$ but not known to be in $\mathsf{P}$, see \cite[Thm. 5.4]{BP2009}. 

While the CG-rank of general polytopes in $ \R^n $ can be arbitrarily large compared to $ n $ (even for $n=2$), Eisenbrand \& Schulz~\cite{ES2003} showed that the CG-rank of a polytope contained in $ [0,1]^n $ is always bounded by $ \orderO(n^2 \log n)$.
Unfortunately, there exist polytopes in $ [0,1]^n $ whose CG-rank grows quadratically in $ n $ (see Rothvo{\ss} \& Sanit{\`a}~\cite{RS2013}).

This motivates the study for situations in which the CG-rank is at most a constant independent of $ n $.  
This question has been recently addressed by Cornu\'ejols \& Lee~\cite{CL2016}.
In their work, given a set $ S \subseteq \{0,1\}^n $, they consider the graph $ \forbiddenvtxgraph{S} $ whose vertices are the points of $ \bar{S} := \{0,1\}^n \setminus S $ where two points are adjacent if they differ in exactly one coordinate. 
Their main result is that if the treewidth of $ \forbiddenvtxgraph{S} $ (denoted $\tw (\forbiddenvtxgraph{S})$) is at most $ 2 $, then the CG-rank of any polytope $ R \subseteq [0,1]^n $ with $ R \cap \{0,1\}^n = S $ is bounded (they prove a bound of $ 4 $, which is tight).
One corollary of our work is that this holds for \emph{all} values of treewidth: the CG-rank of every polytope $R \subseteq [0,1]^n$ with $R \cap \{0,1\}^n=S$ is bounded by a function that only depends on the treewidth of $ \forbiddenvtxgraph{S} $. 

In order to state our main result, we define the \emph{notch} of a subset $ S \subseteq \{0,1\}^n $ as the smallest $ \pitch \in \Z_{\geqslant 0} $ such that every $ p $-dimensional face of $ [0,1]^n $ has a nonempty intersection with $ S $.
If $ S $ is empty, we define $ \pitch := n + 1 $.  
We warn the reader that in a previous version of this paper, we used the term \emph{pitch} instead of notch. We now use the term notch to avoid confusion with 
the definition of pitch due to Bienstock \& Zuckerberg~\cite{BZ2004}.  The difference between pitch and notch is discussed in~\cite{FHW2017}.

We define the \emph{\gap} of $ S $ as the smallest $ \Delta \in \Z_{\geqslant 0} $ such that $ \conv(S) $ can be described as the set of solutions $ x \in \R^n $ satisfying inequalities of the form
\begin{equation} \label{eq:gap}
    \sum_{i \in I} c_i x_i + \sum_{j \in J} c_j (1 - x_j) \geqslant \delta
\end{equation}
where $ I, J $ are disjoint subsets of $ [n] $, $ \delta,c_1,\dotsc,c_n \in \Z_{\geqslant 0} $ with $ \delta \leqslant \Delta $. We require that for each inequality in the description, the corresponding equation (obtained from \eqref{eq:gap} by replacing the inequality sign by an equality sign) defines a hyperplane spanned by $0/1$-points. Notice that if $ S $ is empty, then $ \Delta = 1 $, and if $S=\{0,1\}^n$, then $\Delta=0$.

The gap is well-defined since for every $S\subseteq \{0,1\}^n$, $\conv(S)$ has a description by inequalities in which every corresponding hyperplane is generated by $0/1$-points. To see this, consider a full-dimensional $0/1$- polytope $\conv(T)$ where $S \subseteq T \subseteq \{0,1\}^n$ such that $\conv(S)$ is a face of $\conv(T)$ (this exists since the set $\{0,1\}^n$ is full-dimensional). Clearly, the bounding hyperplane of every facet of $\conv(T)$ is generated by $0/1$-points. Since $\conv(S)$ is the intersection of the facets of $\conv(T)$ which contain it, the claimed description directly follows.

%
%
%
Our main result is the following.
\begin{thm}
    \label{thmMain}
    Let $ S \subsetneq \{0,1\}^n $ be a set with notch $ \pitch $ and \gap{} $ \Delta $.
    Then the CG-rank of any polytope $ R \subseteq [0,1]^n $ with $ R \cap \{0,1\}^n = S $ is at most $ \pitch + \Delta - 1 $.
\end{thm}


In order to generalize the result of Cornu\'ejols \& Lee, we will show that $ \pitch $ and $ \Delta $ are both bounded in terms of $\tw( \forbiddenvtxgraph{S})$.  In fact, we will not even need the definition of treewidth because we actually prove a stronger result. We let $K_t$ be a clique on $t$ vertices. A \emph{subdivision of $K_t$} is a graph obtained from $K_t$ by replacing each edge of $K_t$ by a path.  

\begin{cor} \label{strongergeneralization1}
    Let $ S \subseteq \{0,1\}^n $ and let $ t $ be the maximum integer such that $ \forbiddenvtxgraph{S} $ contains a subdivision of $K_{t+1}$.  
    Then the CG-rank of any polytope $ R \subseteq [0,1]^n $ with $ R \cap \{0,1\}^n = S $ is at most $t + 2t^{t/2}$.
\end{cor}

To see that Corollary~\ref{strongergeneralization1} is a generalization of the result by Cornu\'ejols \& Lee~\cite{CL2016}, the only thing the reader needs to know is that if a graph $G$ has a subdivision of $K_{t+1}$, then $\tw(G) \geqslant t$.  This is an easy fact (see Diestel~\cite{Diestel2010} for a gentle introduction to treewidth). Note that plugging $t=2$ into Corollary~\ref{strongergeneralization1} gives a bound of $6$ instead of $4$ (as obtained in~\cite{CL2016}).  However, as an easy corollary of another of our results (Theorem~\ref{thm:PitchThree}), we also obtain a bound of $4$ when $t=2$ in Corollary~\ref{strongergeneralization1}.

%
%

\subsubsection*{Paper structure}

In Section~\ref{secDiscussionParameters} we discuss the meaning of the parameters $ \pitch $ and $ \Delta $, and in particular their relation to the CG-rank.
For instance, we give examples that show that the CG-rank of a polytope in $ [0,1]^n $ cannot be bounded in only one of the two parameters.
Furthermore, we observe that optimizing a linear function over $ S $ can be done with $ \orderO(n^\pitch) $ oracle calls using an oracle that decides if a point $ x \in \{0,1\}^n $ belongs to $ S $, see Proposition~\ref{propOptimizingSmallPitch}. This algorithm already appears in~\cite{CL2016}, but we show that it is valid under a weaker hypothesis.

Section~\ref{secProof} contains the proof of Theorem~\ref{thmMain}.
In Section~\ref{sec:apx} we complement our main theorem by a result quantifying how well the $t$-th CG-closure approximates $\conv(S)$ for constant $t$ and constant $\pitch$, this time without bounding $\Delta$.
In Section~\ref{secPitchThree} we investigate the convex hulls of sets with notch $\pitch = 3$. In this case, we show that $ \Delta $ is automatically bounded and give a complete linear description of $ \conv(S) $. We show that treewidth at most $2$ implies notch at most $3$, but not vice versa, hence this result also strictly generalizes the main result of Cornu\'ejols \& Lee~\cite{CL2016}.
%

\section{Discussion of the parameters}
\label{secDiscussionParameters}

In this section, we discuss how the parameters \emph{notch} and \emph{\gap} of a set $ S \subseteq \{0,1\}^n $ influence the CG-rank of polytopes $ R \subseteq [0,1]^n $ with $ R \cap \{0,1\}^n = S$.

\subsection{Small CG-rank implies small notch}
We first observe that, in order to get a constant bound on the CG-rank, one has to restrict to sets $ S $ with bounded notch. Although this follows directly from known results, we include a proof for completeness. We point out that a little more work gives a lower bound of $\pitch$, which is tight.
\begin{prop} \label{prop:small_rank_implies_small_pitch}
    Let $ S \subseteq \{0,1\}^n $ have notch $ \pitch $.
    Then there exists a polytope $ R \subseteq [0,1]^n $ with $ R \cap \{0,1\}^n = S $ whose CG-rank is at least $ p - 1$.
\end{prop}
\begin{proof}
Following~\cite{CL2016}, we let $R$ be the worst possible\footnote{In the sense that $R' \supseteq Q'$ for all polytopes $Q \subseteq [0,1]^n$ such that $Q \cap \{0,1\}^n = R \cap \{0,1\}^n = S$.  Note that this implies that $R$ has the highest CG-rank among all such polytopes $Q$.} relaxation of $\conv(S)$:
\begin{equation}
\label{eq:worst_relax}
    R := \left\{ x \in [0,1]^n \mid \forall a \in \bar{S} : \sum_{i : a_i = 0} x_i + \sum_{i : a_i = 1} (1-x_i) \geqslant \tfrac{1}{2} \right\}.
\end{equation}
By the definition of $ \pitch $, there exists a $ (\pitch - 1) $-dimensional face $ F $ of $ [0,1]^n $ such that $ F \cap S = \emptyset $. The CG-rank of $R$ is at least that of its face $F \cap R$ (see for instance~\cite[Lem.~5.17]{CCZ2014}), which can be shown to be at least $ \pitch-1 $ using \cite[Lem.~7.2]{CCH1989}.
\end{proof}

It turns out that the structure of sets $ S \subseteq \{0,1\}^n $ with small notch $ \pitch $ can be efficiently exploited with respect to certain optimization tasks.
For instance, the $ \pitch $-th level of the Bienstock-Zuckerberg hierarchy~\cite{BZ2004} gives a tight description of $ \conv(S) $, at least when applied to sets $ S $ of set-covering type.
A much simpler observation is that linear programming over $ S $ is easy if $ \pitch $ is constant.

\subsection{Optimization algorithm for small notch}
Let $S \subseteq \{0,1\}$ have notch $\pitch$.  Assume that we have an oracle for deciding if a given point $x \in \{0,1\}^n$ belongs to $ S $. Here, we prove that optimizing a linear function over $ S $ can be done after performing at most $\orderO(n^p)$ oracle calls, and spending an extra polynomial time to select an optimum solution.

The algorithm is as follows. Given a cost vector $c \in \R^{n}$, we let $x^* \in \{0,1\}^n$ be defined as $x^*_{i} := 0$ if $c_i \geqslant 0$ and $x^*_{i} := 1$ if $c_i < 0$. Note that this is an optimum solution of $\min \{c^\intercal x \mid x \in \{0,1\}^n\}$. Next, among all the vertices of the cube $x \in \{0,1\}^n$ that are at Hamming distance at most $ \pitch $ from $x^*$, output any vertex $x$ that belongs to $ S $ and has minimum cost. 

\begin{prop}
\label{propOptimizingSmallPitch} For every $S \subseteq \{0,1\}^n$ with notch $\pitch$ and every $c \in \R^n$, the algorithm described above solves $\min \{c^\intercal x  \mid x \in S\}$ in $\orderO(n^{\pitch})$ oracle calls.
\end{prop}

\begin{proof}
Clearly, the number of oracle calls performed by the algorithm is at most
$$
\sum_{k = 0}^{\pitch} \binom{n}{k} \leqslant (n + 1)^{\pitch} = \orderO(n^{\pitch})\,.
$$

There is always a feasible solution $x \in \{0,1\}^n$ at Hamming distance at most $\pitch$ from $x^*$, since otherwise there would exist a $\pitch$-dimensional face of the cube that is disjoint from $S$, which contradicts that the notch of $S$ is $p$. Therefore, the algorithm always outputs some feasible solution.

In order to finish proving the correctness of the algorithm, consider an optimum solution $x^{\mathrm{opt}}$ in $S$ whose Hamming distance $d_H(x^{\mathrm{opt}},x^*)$ to $x^*$ is minimum. Let $I := \{i \in [n] \mid x^{\mathrm{opt}}_i \neq x^*_i\}$ be the set of indices of bits of $x^*$ that are flipped in $x^{\mathrm{opt}}$, so that we can express the optimum value as
$$
\mathrm{OPT} = c^\intercal x^{\mathrm{opt}} = c^\intercal x^* + \sum_{i \in I} |c_i|\,.
$$
Now consider the set $F$ of vertices $x \in \{0,1\}^n$ that are obtained by flipping the bits of $x^*$ indexed by some set $J \subseteq I$. Thus, $F = \{x \in \{0,1\}^n \mid \forall i \in [n] \smallsetminus I : x_i = x^*_i = x^{\mathrm{opt}}_i\}$. Clearly, $F$ is the vertex set of some face of the cube. Every $x \in F$ has cost at most $\mathrm{OPT}$ since we have
$$
c^\intercal x = c^\intercal x^* + \sum_{i \in J} |c_i|
\leqslant c^\intercal x^* + \sum_{i \in I} |c_i| = c^\intercal x^{\mathrm{opt}} = \mathrm{OPT}\,.
$$ 
By minimality of $d_H(x^{\mathrm{opt}},x^*)$, no $x \in F \smallsetminus \{x^{\mathrm{opt}}\}$ belongs to $S$.
Thus, $d_H(x^{\mathrm{opt}},x^*) \leqslant \pitch$ (otherwise, $F$ would contain a $p$-dimensional face of $[0,1]^n$ disjoint from $S$), and $x^{\mathrm{opt}}$ is one of the feasible solutions considered by the algorithm. The result follows.
\end{proof}

\subsection{Small CG-rank implies small gap}

One might wonder whether sets $ S \subseteq \{0,1\}^n $ with small notch are already simple enough to ensure that every relaxation for $ S $ contained in $ [0,1]^n $ has small CG-rank.
However, it turns out that such sets $ S $ also need to have a description with bounded coefficients only, as illustrated by the next two results.
Here, we denote by $||A||_\infty$ the maximum absolute value of an entry of $A$.

\begin{lem} \label{lem:increase}
Let $P = \{x \in \R^n \mid Ax \geqslant b\}$, where $A \in \Z^{m \times n}$ and $b \in \Z^m$. Letting $P'$ denote the first CG-closure of $P$, there is a description $P' = \{x \in \R^n \mid Bx \geqslant c\}$ with $ B $ and $ c $ integer such that $||B||_\infty \leqslant n||A||_\infty$.
\end{lem}

\begin{proof}
Every valid inequality for $P'$ can be written as $\lambda^\intercal A x \geqslant \lceil \lambda^\intercal b \rceil$ for some $\lambda \in \R^m_+$. By Carath\'eodory's theorem, we may assume that $\lambda$ has at most $n$ non-zero entries. Furthermore, it is well known that one can replace every entry of $\lambda$ by its non-integral part to obtain an inequality that is valid for $P'$ and at least as strong as the original one (see, e.g.,~\cite[Lem.~5.13]{CCZ2014}). In other words, we may assume that $\lambda \in [0, 1)^m$ and $\lambda$ has at most $n$ non-zero entries. By the triangle inequality, this implies
$$
||\lambda^\intercal A||_\infty =
\norm{\sum_{i : \lambda_i \neq 0} \lambda_i A_i}_\infty
\leqslant \sum_{i : \lambda_i \neq 0} \underbrace{\lambda_i ||A_i||_\infty}_{\leqslant ||A||_\infty}
\leqslant n ||A||_\infty\,,
$$
and the lemma follows.
\end{proof}
\begin{prop}
    \label{propCGrankLowerBoundedByGap}
    Let $ S \subsetneq \{0,1\}^n $ be nonempty with \gap{} $ \Delta $.
    Then there exists a polytope $ R \subseteq [0,1]^n $ with $ R \cap \{0,1\}^n = S $ such that the CG-rank of $ R $ is at least $ \frac{\log \Delta}{\log n} - 1 $.
\end{prop}
\begin{proof}
    Note that we forbid $S =\{0,1\}^n$, since otherwise $\Delta=0$ and $\log \Delta$ is undefined.  
    First, we claim that every integer matrix $ A $ for which there is an integer vector $ b $ with $ \conv(S) = \{ x \in \R^n \mid Ax \geqslant b \} $ satisfies $ \|A\|_\infty \geqslant \frac{\Delta}{n} $.
    Indeed, every inequality in such a description is of the form
    \[
        \sum_{i \in I} c_i x_i - \sum_{j \in [n] \setminus I} c_j x_j \geqslant \beta
    \]
    where $ I \subseteq [n] $, $ c \in \Z^n_{\geqslant 0} $, $ \beta \in \Z $. Letting $\delta := \beta + \sum_{j \in [n] \setminus I}c_j$ we can rewrite this inequality as
    \[
        \sum_{i \in I} c_i x_i + \sum_{j \in [n] \setminus I} c_j (1-x_j) \geqslant \delta\,.
    \]
    By the definition of $ \Delta $, for at least one of these inequalities we must have $ \delta \geqslant \Delta $.
    Since $ \conv(S) \subseteq [0,1]^n $ is nonempty, this implies $ \|c\|_\infty \geqslant \frac{\Delta}{n} $.
    
    Second, consider the polytope $ R := \{ x \in [0,1]^n \mid \forall a \in \bar{S} : \sum_{i : a_i = 0} x_i + \sum_{j : a_j = 1} (1-x_j) \geqslant 1\} $.
    Note that $ R \subseteq [0,1]^n $ and $ R \cap \{0,1\}^n = S $.
    Furthermore, $ R $ has a description of the form $ R = \{ x \in \R^n \mid Cx \geqslant d \} $ with $ C,d$ integer and $ \|C\|_\infty = 1 $.
    Thus, letting $ k $ be the CG-rank of $ R $ and in view of Lemma~\ref{lem:increase}, we obtain $ n^k \geqslant \frac{\Delta}{n} $, which yields the claim.
\end{proof}
Let $ S \subseteq \{0,1\}^n $ with notch $ \pitch $ and \gap{} $ \Delta $, and denote by $ k $ the largest CG-rank of a polytope $ R \subseteq [0,1]^n $ with $ R \cap S = \{0,1\}^n $.
Summarizing the previous observations, we have seen that $ k $ can be bounded from below in terms of $ \pitch $ (Proposition~\ref{prop:small_rank_implies_small_pitch}), and also in terms of $ \Delta $ and $n$ (Proposition~\ref{propCGrankLowerBoundedByGap}).
This explains the occurrence of both parameters in the statement of Theorem~\ref{thmMain}.

In what follows next, we would like to discuss that none of the two parameters $ \pitch $ and $ \Delta $ can be bounded by a function that only depends on the other.
To see that $ \pitch $ cannot be bounded by a function in $ \Delta $, observe that the set $ S = \{x \in \{0,1\}^n \mid x_{\pitch} + x_{\pitch + 1} + \cdots + x_n \geqslant 1 \}$ has notch $ \pitch $ and \gap{} $ 1 $.

\subsection{Bounded Notch Does Not Imply Bounded CG-rank} \label{sec:bd_pitch_bad}
Next, we show that neither the parameter $ \Delta $ nor the CG-rank can be bounded in terms of $ \pitch $ alone. 

\begin{prop} \label{prop:badfacet}
For each $n \in \mathbb{N}$, there exists $S_n \subseteq \{0,1\}^{2n+2}$ such that $S_n$ has notch at most $7$ but gap at least $2^{n+1}$.  
\end{prop}
\begin{proof}
Fix $n \in \mathbb{N}$.  We define a vector $c \in \R^{2n+2}$ by setting $c_1=2^n$, $c_2=2^{n-1}$, $c_i=c_{i-1}$ if $i \in [3,2n+1]$ is odd, $c_i=(2^n-c_{i-1})/2$ if $i \in [3,2n+1]$ is even, and $c_{2n+2}=2^n-c_{2n+1}$.  Now consider the inequality $\sum_{i=1}^{2n+2} c_i x_i \geqslant 2^{n+1}$, and let $S_n$ be the set of vectors in $\{0,1\}^{2n+2}$ for which this inequality is satisfied. 

By definition, $\sum_{i=1}^{2n+2} c_i x_i \geqslant 2^{n+1}$ is a valid inequality for $\conv(S_n)$.  We claim that it is actually a facet of $\conv(S_n)$. This follows by observing that $c_1+c_2+c_3=2^{n+1}$, $c_1+c_{2i-1}+c_{2i}+c_{2i+1}=c_1+c_{2i-2}+c_{2i}+c_{2i+1}=2^{n+1}$ for all $i \in [2,n]$, $c_1+c_{2n}+c_{2n+2}=c_1+c_{2n+1}+c_{2n+2} = 2^{n+1}$ and $c_2+c_3+c_{2n+1}+c_{2n+2} = 2^{n+1}$.

 By solving a linear recurrence of degree~$1$, we find that $c_{2i} = c_{2i+1} = 2^n \cdot (1 - (-1/2)^i)/3$ for $i \in [1,n]$. It follows that 
the greatest common divisor of the entries of $c$ is $1$ since $c_1$ is a power of $2$ and $c_{2n}$ is odd.  Since all $c_i$ are non-negative, this implies that the gap of $S_n$ is at least $2^{n+1}$.

Finally, we show that $S_n$ has notch at most $7$.  That is, we must show that the $7$ smallest entries of $c$ sum to at least $2^{n+1}$.  This is easily checked by hand if $n<8$, so we may assume $n \geqslant 8$. Since $c_{2i} = c_{2i+1} = 2^n \cdot (1 - (-1/2)^i)/3$ for $i \in [1,n]$, it follows that the $7$ smallest entries of $c$ are $c_4, c_5, c_8, c_9, c_{12}, c_{13}$, and $c_{16}$.   The sum of these entries is  
\[
2^n \cdot \left(\frac{1}{4}+\frac{1}{4}+\frac{5}{16}+\frac{5}{16}+\frac{21}{64}+\frac{21}{64}+\frac{85}{256}\right) = 2^n \cdot \frac{541}{256} > 2^n \cdot 2 = 2^{n+1},
\]
as required.
\end{proof}
By Proposition~\ref{propCGrankLowerBoundedByGap} and Proposition~\ref{prop:badfacet} we directly obtain.
\begin{cor} \label{cor:badfacet}
    For each $n$, there exists a polytope $ R \subseteq [0,1]^{2n+2} $ such that $ S = R \cap \{0,1\}^{2n+2} $ has notch at most $ 7 $, but the CG-rank of $ R $ is $ \Omega(\frac{n}{\log n}) $.
\end{cor}

\subsection{Bounded Treewidth Implies Bounded Notch and \Gap}

Finally, we demonstrate that Theorem~\ref{thmMain} can indeed be seen as a generalization of the results of Cornu\'ejols \& Lee~\cite{CL2016}.
To this end, it suffices to show that $ \pitch $ and $ \Delta $ can be bounded in terms of the treewidth of $ \forbiddenvtxgraph{S} $.
Recall that the largest $ t $ such that $\forbiddenvtxgraph{S}$ contains a subdivision of $K_{t+1}$ is at most  $ \tw (\forbiddenvtxgraph{S})$. 
This observation, together with the following lemma, imply Corollary~\ref{strongergeneralization1}.
\begin{lem} \label{lem:treewidth}
Let $S \subseteq \{0,1\}^n$, and let $\pitch$ and $\Delta$ respectively denote the notch and the \gap{} of $S$. If $t$ is maximum such that $\forbiddenvtxgraph{S}$ contains a subdivision of $K_{t+1}$, then $\pitch \leqslant t+1$ and $\Delta \leqslant 2t^{t/2}$.
\end{lem}

\begin{proof}
Note that the $d$-dimensional cube contains a subdivision of $K_{d+1}$, where the branch vertices are the vectors with support at most $1$, and the subdivision vertices are the vectors with support $2$.  Now, since $\forbiddenvtxgraph{S}$ contains a subgraph isomorphic to the $(\pitch-1)$-dimensional cube, it contains a subdivision of $K_p$ and we have $t \geqslant \pitch - 1$.


To show $\Delta \leqslant 2t^{t/2}$, observe that it suffices to prove the following. For any hyperplane $H := \{x \in \R^n \mid \sum_{i \in I} c_i x_i + \sum_{j \in [n] \setminus I} c_j (1-x_j) = 1 \}$ that is spanned by $0/1$-points such that $\sum_{i \in I} c_i x_i + \sum_{j \in [n] \setminus I} c_j (1-x_j) \geqslant 1$ is valid for $S$ and $c_1, \ldots, c_n \in \Q_{\geqslant 0}$, there exists some integer number $ K \in [1,2 t^{t/2}]$ such that every $ c_i $ is an integer multiple of $ 1/K $.

By switching the coordinates indexed by $[n] \setminus I$, we may assume that $ I = [n] $.
Define $I_{<1/2} := \{ i \in [n] \mid c_i < 1/2 \}$, and $ I_{=1/2} $, $ I_{>1/2} $ similarly. We have that $|I_{<1/2}| \leqslant t$ since otherwise $\forbiddenvtxgraph{S}$ contains a subdivision of a clique of size $t+2$ whose branch vertices are the characteristic vectors of the empty set $\emptyset$ and the singletons $\{i\}$ for $i \in I_{<1/2}$ and whose subdivision vertices are the characteristic vectors of the pairs $\{i,j\}$ for $i, j \in I_{<1/2}$.


Let $ x \in \{0,1\}^n \cap H $ and denote by $ T $ its support.
Then one of the following holds:
(i) $ |T \cap I_{=1/2}| = |T \cap I_{>1/2}| = 0 $,
(ii) $ |T \cap I_{=1/2}| = 1 $ and $|T \cap I_{>1/2}| = 0 $,
(iii) $ |T \cap I_{=1/2}| = 0 $ and $|T \cap I_{>1/2}| = 1 $,
 or (iv) $ |T \cap I_{=1/2}| = 2 $ and $|T \cap I_{>1/2}| = 0 $.
Thus, the vector $ c $ is the unique solution of a system of linear equations of the following form
\[
    \begin{pmatrix}
        A &     &  \\
        B & *   &  \\
        C &     & D \\
          & \id &  \\
    \end{pmatrix} c = b,
\]
where the coefficient matrix has integer entries,
$ A,B,C $ are $ 0/1 $-matrices with columns indexed by $ I_{<1/2} $, $ \id $ is an identity matrix with columns indexed by $ I_{=1/2} $, $ D $ is a $ 0/1 $-matrix with columns indexed by $ I_{>1/2} $ and exactly one $ 1 $ per row, and $ b $ is a column vector with entries in $ \{1/2, 1\} $.   We have used the convention that $*$-entries can have arbitrary values (that we do not care about) and blank entries always have value $0$.  The last rows of the above system are meant to be the trivial equations $c_i = 1/2 $, which are obviously valid for all $ i \in I_{=1/2} $.

Since every row in $ D $ contains exactly one $ 1 $, we can perform elementary row operations to obtain an equivalent system of the form
\[
        \begin{pmatrix}
        E  &   * &  \\
        *  &     & \id \\
           & \id &
    \end{pmatrix} c = b',
\]
where the coefficient matrix has integer entries,
$ E $ is a matrix with entries in $ \{-1,0,1\} $ and columns indexed by $ I_{<1/2} $, and $ b' $ a column vector with entries in $ \{ 0,1/2,1 \} $.
By removing some rows in the topmost block, we may assume that the coefficient matrix is an invertible $ n \times n $-matrix whose determinant is $ \pm \det(E) $.
Thus, by Cramer's rule, every $ c_i $ is an integer multiple of $ \frac{1}{2 |\det(E)|} $.
Since $ E $ is a matrix with entries in $ \{-1,0,1\} $ and $|I_{<1/2}| \leqslant t $ columns and rows, by the Hadamard bound we obtain
\[
    K = 2 |\det(E)| \leqslant 2 t^{\frac{1}{2}t},
\]
as claimed. 
\end{proof}

\section{Proof of Main Theorem}
\label{secProof}

\begin{lem}
    \label{lemCGRankCoverInequalities}
    Let $ R \subseteq [0,1]^n $ be a polytope and $ I,J \subseteq [n] $ with $ I \cap J = \emptyset $ such that
    \begin{equation}
        \label{eqInequalityCutOffFace}
        \sum_{i \in I} x_i + \sum_{j \in J} (1-x_j) \geqslant 1
    \end{equation}
    holds for every $ x \in R \cap \{0,1\}^n $.
    Then~\eqref{eqInequalityCutOffFace} is also valid for $ R^{(n + 1 - (|I| + |J|))} $.
\end{lem}
\begin{proof}
    Consider the set
    \[
        F := \{ x \in R \mid x_i = 0 \ \forall \, i \in I, \, x_j = 1 \ \forall \, j \in J \},
    \]
    which is a face of $ R $ of dimension $ k \leqslant n - (|I| + |J|) $.
    Since~\eqref{eqInequalityCutOffFace} is valid for $ R \cap \{0,1\}^n $, we have $ F \cap \Z^n = \emptyset $.
    Since $ F \subseteq [0,1]^n $, this implies $ F^{(k)} = \emptyset $ (by~\cite[Lem.~2.2]{ES2003}).
    This implies $ R^{(k)} \cap F = \emptyset $ (see~\cite[Lem.~5.17]{CCZ2014}) and hence there exists an $ \varepsilon > 0 $ such that
    \[
        \sum_{i \in I} x_i + \sum_{j \in J} (1-x_j) \geqslant \varepsilon
    \]
    is valid for $ R^{(k)} $.  This means that~\eqref{eqInequalityCutOffFace} holds for $ R^{(k + 1)} $, as claimed.
\end{proof}

\begin{proof}[Proof of Theorem~\ref{thmMain}]
By the definition of $ \Delta $, we can find a description of $ \conv(R \cap \{0,1\}^n) $ by means of linear inequalities where every inequality is of the form
\begin{equation}
    \label{eqTargetInequality}
    \sum_{i \in I} c_i x_i + \sum_{j \in J} c_i (1-x_i) \geqslant \delta
\end{equation}
for some $ I,J \subseteq [n] $ with $ I \cap J = \emptyset $, where $ \delta \in \Z_{\geqslant 0} $, $ c_i \in \Z_{\geqslant 1} $ for all $ i \in I \cup J $, and $ \delta \leqslant \Delta $.
Note that every such inequality with $ \delta = 0 $ is already valid for $ R $.
For inequalities with $ \delta \geqslant 1 $, we may assume that $ c_i \leqslant \delta $ holds for every $ i \in I \cup J $.

By induction on $ \delta \geqslant 1 $ we will show that~\eqref{eqTargetInequality} holds for every $ x \in R^{(\pitch + \delta - 1)} $, which then yields the claim.
If $ \delta = 1 $, then we have $ c_i = 1 $ for all $ i \in I \cup J $.
By Lemma~\ref{lemCGRankCoverInequalities}, we know that Inequality~\eqref{eqTargetInequality} is valid for $ R^{(t)} $, where $ t = n + 1 - (|I| + |J|) $.
It remains to show that $ t \leqslant \pitch + \delta - 1 = \pitch $.
To this end, consider the set
\[
    F = \{ x \in [0,1]^n \mid x_i = 0 \ \forall \, i \in I, \, x_j = 1 \ \forall \, j \in J \},
\]
which is a face of the cube, and note that no point of $ F $ satisfies~\eqref{eqTargetInequality}.
Thus, we indeed obtain $ \pitch \geqslant \dim(F) + 1 = n + 1 - (|I| + |J|) = t $.

Now let $ \delta \geqslant 2 $.
We may assume that $ |I| + |J| \geqslant 1 $, otherwise we can divide~\eqref{eqTargetInequality} by $ \delta $ and proceed by induction.
For every $ i_0 \in I $ consider the inequality
\[
    \sum_{i \in I \setminus \{i_0\}} c_i x_i + (c_{i_0} - 1) x_{i_0} + \sum_{j \in J} c_j (1 - x_j) \geqslant \delta - x_{i_0} \geqslant \delta - 1,
\]
which is valid for $ R \cap \{0,1\}^n $.
Similarly, for every $ j_0 \in J $ the inequality
\[
    \sum_{i \in I} c_i x_i + \sum_{j \in J \setminus \{j_0\}} c_j (1 - x_j) + (c_{j_0} - 1) (1 - x_{j_0}) \geqslant \delta - (1 - x_{j_0}) \geqslant \delta - 1
\]
is also valid for $ R \cap \{0,1\}^n $.
Thus, by the induction hypothesis, both such inequalities are valid for $ R^{(\pitch + \delta - 2)} $.
Summing these $ k := |I| + |J| $ many inequalities up and dividing them by $ k \geqslant 1 $, we obtain that
\[
    \sum_{i \in I} \left( c_i - \frac{1}{k} \right) x_i + \sum_{j \in J} \left( c_j - \frac{1}{k} \right) (1 - x_j) \geqslant \delta - 1
\]
is valid for $ R^{(\pitch + \delta - 2)} $.
Choose $ \varepsilon > 0 $ such that $ (1 + \varepsilon) (c_i - \frac{1}{k}) \leqslant c_i $ holds for all $ i \in I \cup J $.
Scaling the above inequality by $ (1 + \varepsilon) $, we thus obtain that
\begin{align*}
    \sum_{i \in I} c_i x_i + \sum_{j \in J} c_j (1 - x_j)
    & \geqslant (1+\varepsilon) \left( \sum_{i \in I} \left( c_i - \frac{1}{k} \right) x_i + \sum_{j \in J} \left( c_j - \frac{1}{k} \right) (1 - x_j) \right) \\
    & \geqslant (1 + \varepsilon) (\delta - 1),
\end{align*}
holds for every $ x \in R^{(\pitch + \delta - 2)} $, and hence
\[
    \sum_{i \in I} c_i x_i + \sum_{j \in J} c_j (1 - x_j)
    \geqslant \lceil (1 + \varepsilon) (\delta - 1) \rceil \geqslant \delta
\]
is valid for $ R^{(\pitch + \delta - 1)} $, as claimed.
\end{proof}

\section{Approximating the Integer Hull when the Notch is Bounded} \label{sec:apx}

We have shown in Section~\ref{sec:bd_pitch_bad} that if we only assume that $\pitch$ is constant, it might take $\Omega(n / \log n)$ rounds of CG-cuts to converge to the integer hull: we have to control $\Delta$ also in order to guarantee bounded CG-rank. Here we prove that bounding $\pitch$ alone is in fact enough to obtain good \emph{approximations} of the integer hull after a bounded number of rounds. This is in contrast with the results of Singh \& Talwar \cite{ST2010}, who show that for many problems performing a constant number of rounds of CG-cuts does not significantly decrease the integrality gap.

\begin{cor} \label{corMain}
Let $S \subseteq \{0,1\}^n$ have notch~$\pitch$ and let $\varepsilon \in (0,1)$ be such that $\pitch \varepsilon^{-1} \in \Z_{\geqslant 0}$. For every $t \geqslant \pitch \varepsilon^{-1} - 1$ and for every inequality $\sum_{i \in I} c_i x_i + \sum_{j \in J} c_j (1 - x_j) \geqslant \delta$ that is valid for $\conv(S)$ with $\delta \geqslant c_1, \ldots, c_n \geqslant 0$, the inequality $\sum_{i \in I} c_i x_i + \sum_{j \in J} c_j (1 - x_j) \geqslant (1-\varepsilon)\delta$ is valid for $R^{(t)}$, where $R \subseteq [0,1]^n$ is any polytope such that $R \cap \{0,1\}^n = S$.
\end{cor}
\begin{proof}
After flipping some coordinates, we may assume that $J = \emptyset$. After scaling, we may further assume that $\delta = 1$. Let $K := \pitch \varepsilon^{-1}$. Consider the valid inequality $\sum_{i \in I} \tilde{c}_i x_i \geqslant \tilde{\delta}$ where $\tilde{c}_i := \frac{1}{K} \lfloor K c_i \rfloor \in \{0,1/K,2/K,\ldots,1\}$ and $\tilde{\delta} := \min \{\sum_{i \in I} \tilde{c}_i x_i \mid x \in S\}$. We claim that $\tilde{\delta} \geqslant 1-\varepsilon$. Indeed, let $x \in S$ be arbitrary and let $y \in S$ be such that $0 \leqslant y \leqslant x$ and $y$ has support on at most $\pitch$ coordinates. Then
\begin{align*}
\sum_{i \in I} \tilde{c}_i x_i
&\geqslant \sum_{i \in I} \tilde{c}_i y_i\\
&= \sum_{i \in I} \frac{1}{K} \lfloor K c_i \rfloor y_i\\
&\geqslant \sum_{i \in I} \frac{1}{K} (K c_i - 1) y_i\\
&\geqslant \sum_{i \in I} c_i y_i - \frac{\pitch}{K}\\
&\geqslant 1 - \frac{\pitch}{K} = 1 - \varepsilon\,
\end{align*}
so that $\tilde{\delta} \geqslant 1-\varepsilon$. Now consider the valid inequality $\sum_{i \in I} K \tilde{c}_i x_i \geqslant K (1 - \varepsilon) = K - \pitch$ with nonnegative integer coefficients. From the proof of Theorem~\ref{thmMain}, we see that this inequality is valid for the $t$-th CG-closure of $R$ since $t \geqslant K - 1 = (K - \pitch) + \pitch - 1$.
\end{proof}

\section{The notch-$3$ case}
\label{secPitchThree}

\begin{thm} \label{thm:PitchThree}
Let $S \subseteq \{0,1\}^n$ have notch $\pitch \leqslant 3$. Then $P = \conv(S)$ can be defined by $0 \leqslant x_i \leqslant 1$ for $i \in [n]$ together with inequalities that can be brought in the following form after flipping some coordinates, where for each inequality the subsets of indices are a partition of $[n]$ (we allow empty sets in the partition):
%
\begin{alignat}{13}
\label{eqFacetType1} \sum_{i \in I_0} 0 x_i &\; +\; &\sum_{i \in I_1} 1 x_i &&&&&&&&&\geqslant 1,\quad & |I_0| = 2\\
\label{eqFacetType2} \sum_{i \in I_0} 0 x_i &\; +\;  &\sum_{i \in I_1} 1 x_i &\; +\; &\sum_{i \in I_2} 2 x_i &&&&&&&\geqslant 2, & |I_0| \leqslant 1\\
\label{eqFacetType3} &&\sum_{i \in I_1} 1 x_i &\; +\; & \sum_{i \in I_2} 2 x_i &\; +\; & \sum_{i \in I_3} 3 x_i &&&&&\geqslant 3, & |I_1|\geqslant3\\
\label{eqFacetType4} &&\sum_{i \in I_1} 1 x_i &\; +\; &\sum_{i \in I_2} 2 x_i &\; +\; & \sum_{i \in I_3} 3 x_i &\; +\; &\sum_{i \in I_4} 4 x_i &&&\geqslant 4, &|I_1|=2, |I_2| \geqslant 1\\
\label{eqFacetType6} &&&&\sum_{i \in I_2} 2 x_i &\; +\; &\sum_{i \in I_3} 3 x_i &\; +\; &\sum_{i \in I_4} 4 x_i &\; +\; &\sum_{i \in I_6} 6 x_i &\geqslant 6, & |I_2| \geqslant 3
\end{alignat}
In particular, $S$ has \gap{} $\Delta \leqslant 6$.
\end{thm}
\begin{proof}
We may assume that $n \geqslant 3$, otherwise the theorem holds trivially. Thus, $S$ is nonempty. Pick any nonredundant inequality description of $\conv(S)$ such that the corresponding hyperplanes are spanned by $0/1$-points. Let $(c^{*})^\intercal x \geqslant \delta$ be any inequality in this description which is not of the form $x_i \geqslant 0$ or $1-x_i \geqslant 0$. By flipping coordinates and scaling we may assume that $c_i^* \in \Q_{\geqslant 0}$ and $\delta=1$.  We choose a non-redundant system that uniquely defines $c^*$ consisting of equations of the form $c_i=0, c_i-c_j=0$, and $\sum_{i \in I \subseteq [n]} c_i=1$ such that equations of lower support are always included before equations of higher support.  In particular, this implies that equations of the form $c_i=0$ or $c_i=1$ are always included if $c_i^*=0$ or $c_i^*=1$.

Sort the entries of $c^*$ as $c_1^* \leqslant c_2^* \leqslant \dots \leqslant c_n^*$.  Clearly, since $S$ has notch at most $3$,
$c_1^*+c_2^*+c_3^* \geqslant 1$. Hence, no equation with support greater than $3$ is valid (since the $c_i^*$ are sorted). If $c_1^*+c_2^*+c_3^*=1$, then any equation whose support has size 3
is already implied by the equation $c_1+c_2+c_3=1$ together with equations of the form $c_i-c_j=0$ for $i \in [3]$.  If $c_1^*+c_2^*+c_3^*>1$, then no equation of the form $c_i + c_j + c_k = 1$ will appear.  Thus, at most one equation of support $3$ appears.  

 Define a graph $G=([n], E)$, where $E:=\{ij: c_i+c_j=1 \text{ or } c_i-c_j=0\}$. Let $\Sigma:=\{ij: c_i+c_j=1\}$ and define a cycle of $G$ to be \emph{unbalanced} if it contains an odd number of edges of $\Sigma$.  
Since the system is non-redundant, each component of $G$ contains at most one cycle, which will have to be unbalanced.
For each $\gamma \in [0,1]$, let $J_\gamma:=\{i \in [n] : c_i^*=\gamma\}$.  Note that $|J_0| \leqslant 2$, as $c_1^*+c_2^*+c_3^*\geqslant 1$. 
Let $J_{\frac{1}{2}}'$ be the set of vertices of $G$ contained in a component with an unbalanced cycle.  Clearly, $J_{\frac{1}{2}}' \subseteq J_{\frac{1}{2}}$.

Let $T_1, \dots T_\ell$ be the components of $G$ which contain at least one edge and no cycles. Note that if $\ell \geqslant 2$, then the set of solutions of the system obtained by removing the single equation of the form $c_i + c_j + c_k = 1$ (if it exists) has dimension at least $2$.  Thus, the solution set of the full system has dimension at least $1$, which contradicts the uniqueness of $c^*$.  Therefore, $\ell \leqslant 1$. We may partition the vertices of $T_1$ as $J_{\alpha}' \cup J_\beta'$ where $c_i^*=\alpha$ for all $i \in J_\alpha'$ and $c_i^*=1-\alpha:=\beta$ for all $i \in J_\beta'$.  Note that if $\alpha=0$, then $J_\alpha' \subseteq J_0$ and $J_\beta' \subseteq J_1$, and if $\alpha=\frac{1}{2}$, then $J_{\alpha}' \cup J_{\beta}' \subseteq J_{\frac{1}{2}}$.   
 
It follows that $[n]:=J_0 \cup J_\alpha \cup J_{\frac{1}{2}} \cup J_{\beta} \cup J_1$, for some $0<\alpha<\frac{1}{2}$ and $\beta:=1-\alpha$ (some of these sets are possibly empty).  There are now various cases to consider depending on where the indices of the single equation $c_i + c_j + c_k = 1$ (if it exists) belong.  

First suppose that there does not exist an equation of the form $c_i + c_j + c_k = 1$. 
In this case, by the uniqueness of $c^*$, we must have $J_\alpha=J_\beta=\emptyset$.  If $|J_0|=2$, then $J_{\frac{1}{2}}=\emptyset$ and $J_{1} \neq \emptyset$, so we get (\ref{eqFacetType1}) with $(I_0,I_1)=(J_0,J_1)$.  If $|J_0| \leqslant 1$, we get (\ref{eqFacetType2}) with $(I_0, I_1, I_2)=(J_0,J_{\frac{1}{2}},J_1)$.

We may hence assume there does exist an equation of the form $c_i + c_j + c_k = 1$ (with $i<j<k$).  We may further assume that $\{i,j,k\} \cap J_0=\emptyset$, because otherwise, the equation $c_i + c_j + c_k = 1$ is implied by the lower support equations $c_i=0$ and $c_j+c_k=1$.  Similarly, $\{i,j,k\} \cap J_1=\emptyset$.

Suppose $\{i,j,k\} \subseteq J_\alpha$. This implies that $\alpha=\frac{1}{3}$ and, since $c_1^*+c_2^*+c_3^*\geqslant 1$, $J_0=\emptyset$ and $|J_{\frac{1}{3}}|\geqslant 3$. If $J_{\frac{1}{2}}=\emptyset$ then we get (\ref{eqFacetType3}) with $(I_1,I_2,I_3)=(J_{\frac{1}{3}}, J_{\frac{2}{3}}, J_1)$. If $J_{\frac{1}{2}} \neq \emptyset$, then we get (\ref{eqFacetType6}) with $(I_2,I_3,I_4,I_6)=(J_{\frac{1}{3}}, J_{\frac{1}{2}}, J_{\frac{2}{3}}, J_1)$.

Suppose $\{i,j\}\subseteq J_\alpha$ and $k\in  J_{\frac{1}{2}}$. This implies $\alpha=\frac{1}{4}$, and since $c_1^*+c_2^*+c_3^*\geqslant 1$, we have $J_0=\emptyset$, $|J_\alpha|=2$, and $|J_{\frac{1}{2}}| \geqslant 1$. So, we get (\ref{eqFacetType4}) with  $(I_1,I_2,I_3,I_4)=(J_{\frac{1}{4}}, J_{\frac{1}{2}}, J_{\frac{3}{4}}, J_1)$.

Suppose $\{i,j\}\subseteq J_\alpha$ and $k \in J_{\beta}$. This implies $2\alpha+(1-\alpha) = 1$, and so $\alpha=0$.   This contradicts $\alpha>0$.  

Finally if $|\{i,j,k\} \cap J_{\alpha}| \leqslant 1$, then $c_i^*+c_j^*+c_k^*>1$, which is a contradiction.
\end{proof}

Applying Theorem~\ref{thmMain}, we obtain the following result.

\begin{cor}
    Let $ S \subseteq \{0,1\}^n $ be a set with notch at most $3$.
    Then the CG-rank of every polytope $R \subseteq [0,1]^n $ with $R \cap \{0,1\}^n=S$ is at most $8$.
\end{cor}

Note that when $\tw(\forbiddenvtxgraph{S}) \leqslant 2$, none of the inequalities \eqref{eqFacetType3}, \eqref{eqFacetType4}, or \eqref{eqFacetType6} can appear in the linear description of $\conv(S)$ because for each of them there is a set of indices $I \subseteq [n]$ of size~$3$ such that the characteristic vector of every proper subset of $I$ is in $\bar{S}$. This implies that $\forbiddenvtxgraph{S}$ contains a subdivision of $K_4$. Hence, we recover the same upperbound of $4$ on the CG-rank when $\tw(\forbiddenvtxgraph{S}) \leqslant 2$ established by Cornu\'ejols \& Lee~\cite{CL2016}. On the other hand, the notch $3$ case includes graphs of unbounded treewidth. For example, if we let $S \subseteq \{0,1\}^n$ be the set of vectors of support at least $3$, then $S$ has notch $3$ and $\forbiddenvtxgraph{S}$ contains a subdivision of $K_{n+1}$.   

It is an interesting open question whether $\Delta$ is also bounded by a constant when $p \in \{4,5,6\}$ (we know that $\Delta$ can be unbounded when $p=7$ by Corollary~\ref{cor:badfacet}).  When $p=3$, we showed that the coefficients could be described by using at most one equation of support more than $2$. Therefore, the rest of the equations could be encoded by an (edge-coloured) graph.  A key observation was that the components of this graph are particularly easy to describe. However, this breaks down if we attempt to carry out the same strategy for $p \in \{4,5,6\}$, as we have to use hypergraphs instead of graphs.

\section{Acknowledgments}

We thank Pierre Aboulker, Yuri Faenza and Robert Weismantel for discussions at an early stage of this research. We also thank Gennadiy Averkov, Michele Conforti and Volker Kaibel for their feedback.  Finally, we would like to thank two anonymous referees for their careful reading and very constructive comments on the paper. Yohann Benchetrit, Samuel Fiorini and Tony Huynh were supported by ERC Consolidator Grant 615640-ForEFront.

\bibliography{references}{}
\bibliographystyle{amsplain}

\end{document}